\documentclass[12pt,reqno]{amsart}
\usepackage{amsfonts}
\usepackage{amssymb,color}
\usepackage{amssymb}

\def\jb#1{\langle#1\rangle}
\def\norm#1{\|#1\|}
\def\normo#1{\left\|#1\right\|}
\newcommand{\les}{{\lesssim}}
\newcommand{\ges}{{\gtrsim}}

\newcommand{\F}{\mathcal{F}}

\newcommand{\cP}{\mathcal{P}}
\newcommand{\cU}{\mathcal{U}}
\newcommand{\cN}{\mathcal{N}}

\newcommand{\R}{\mathbb{R}}
\newcommand{\Z}{\mathbb{Z}}

\newcommand{\al}{\alpha}

\newcommand{\e}{\varepsilon}

\newcommand{\om}{\omega}

\newcommand{\x}{\xi}
\newcommand{\y}{\eta}

\newcommand{\ro}{\rho}
\newcommand{\ft}{{\mathcal{F}}}

\newcommand{\De}{\Delta}
\newcommand{\Om}{\Omega}

\newcommand{\p}{\partial}
\newcommand{\na}{\nabla}
\newcommand{\re}{\mathop{\mathrm{Re}}}

\newcommand{\lec}{\lesssim}

\newcommand{\I}{\infty}

\newcommand{\ti}{\widetilde}

\newcommand{\LR}[1]{{\langle #1 \rangle}}

\newcommand{\EQ}[1]{\begin{equation}\begin{split} #1 \end{split}\end{equation}}
\newcommand{\Del}[1]{}
\newcommand{\CAS}[1]{\begin{cases} #1 \end{cases}}

\newcommand{\pt}{&}
\newcommand{\pr}{\\ &}
\newcommand{\pq}{\quad}

\numberwithin{equation}{section}

\newtheorem{thm}{Theorem}[section]

\newtheorem{lem}[thm]{Lemma}

\theoremstyle{remark}
\newtheorem{rem}{Remark}

\begin{document}
%\subjclass[2010]{35L70, 35Q55}
%\keywords{Nonlinear wave equation, nonlinear Schr\"odinger equation}

\title[small energy scattering for Klein-Gordon-Zakharov]
{Small energy scattering\\ for the Klein-Gordon-Zakharov system\\ with
radial symmetry}
\author[Z. Guo, K. Nakanishi, S. Wang]{Zihua Guo, Kenji Nakanishi, Shuxia Wang}

\address{LMAM, School of Mathematical Sciences, Peking
University, Beijing 100871, China}
\address{Beijing International Center for Mathematical Research, Beijing
100871, China} \email{zihuaguo@math.pku.edu.cn}

\address{Department of Mathematics, Kyoto University, Kyoto 606-8502,
Japan}

\email{n-kenji@math.kyoto-u.ac.jp}

\address{LMAM, School of Mathematical Sciences, Peking
University, Beijing 100871, China}

\email{wangshuxia@pku.edu.cn}

\begin{abstract}
We prove small energy scattering for the 3D Klein-Gordon-Zakharov
system with radial symmetry. The idea of proof is the same as the
Zakharov system studied in \cite{GN}, namely to combine the normal
form reduction and the radial-improved Strichartz estimates.
\end{abstract}

\maketitle

\section{Introduction}
In this paper, we consider the Cauchy problem for the 3D  Klein-Gordon-Zakharov
system \EQ{\label{eq:KGZ1}
 \CAS{ \ddot{u} - \De u +u= nu,\\
   \ddot n/\al^2 - \De n = -\De u^2,}}
with the initial data
\begin{align}
u(0,x)=u_0,\,\dot{u}(0,x)=u_1,\  n(0,x)=n_0,\,\dot n(0,x)=n_1,
\end{align}
where $(u,n)(t,x):\R^{1+3}\to\R\times\R$, and $\al>0, \al\neq 1$ denotes the
ion sound speed. It preserves  the energy \EQ{
 E=\int_{\R^3}|u|^2+|\na u|^2+|\dot{u}|^2+\frac{|D^{-1}\dot n|^2/\al^2+|n|^2}{2}-n|u|^2 dx,}
where $D:=\sqrt{-\De}$, as well as the radial symmetry
\EQ{
 (u,n)(t,x)=(u,n)(t,|x|).}
We consider those solutions with such symmetry and finite energy,
hence \EQ{\label{eq:indata}
 (u_0,u_1,n_0,n_1)\in H^1_r(\R^3)\times L^2_r(\R^3)\times L^2_r(\R^3)\times \dot H^{-1}_r(\R^3).}
We are interested in the scattering for small data in the above
function space.

This system describes the interaction between Langmuir waves and ion
sound waves in a plasma (see \cite{B}, \cite{D}). The local
well-posedness (for arbitrary initial data) and global well-posedness (for small initial data) of \eqref{eq:KGZ1} with $\al<1$ in the energy space $H^1\times
L^2$ was proved by Ozawa, Tsutaya and Tsutsumi in \cite{OTT}.
We point out that \eqref{eq:KGZ1} does not have null form
structure as in Klainerman and Machedon \cite{KM} and this suggests
that when $\alpha=1$ the system \eqref{eq:KGZ1} may be locally
ill-posed in $H^1\times L^2$ (cf. the counter example of Lindblad
\cite{L} for similar equations). Hence, we suppose $\al\neq 1$ here.
When the first equation of \eqref{eq:KGZ1} is replaced by
$c^{-2}\ddot{u} - \De u +c^2u= -nu$ and $c, \alpha\rightarrow\I$,
Masmoudi and Nakanishi studied the limit system and the behavior of
their solutions in a series of papers \cite{MN1}-\cite{MN3}. The
instability of standing wave of Klein-Gordon-Zakharov system was
studied in \cite{GGZ}, \cite{GZ} and \cite{OT}.

In this paper, inspired by \cite{GN}, we combine the normal form
technique, which was first used in a dispersive PDE context by
Shatah \cite{Shatah}, and the improved radial Strichartz estimates
to prove small energy scattering of \eqref{eq:KGZ1} with radial
symmetry.  The normal form transform was also used in \cite{OTT2} for \eqref{eq:KGZ1} and they got the scattering from
initial data small in the Sobolev spaces with high regularity and in $L^p$ with $p<2$. Moreover, their scattering result is independent of  radial symmetry.
The
main result of this paper is
\begin{thm}\label{thm}
If $(u_0,u_1,n_0,n_1)$ are all radial and small enough in the norm of
\eqref{eq:indata}, then the solution $(u,n)$ scatters in this space
as $t\to\pm\I$.
\end{thm}

The main difficulties for the proof of scattering are derivative
loss and slow dispersion of the wave equation together with the
quadratic nonlinearity. The loss of derivative can be overcome by
the normal form transform (under the assumption $\alpha\ne 1$, so we
have good nonlinear structures mainly due to the different
propagation speed.) To handle the quadratic interaction, we have to
assume radial symmetry so that we have wider class of Strichartz
estimates.

\section{Transform of equation}

This section is devoted to transform the equation by using the
normal form.  It is convenient first to change the system into first
order as usual. Let \EQ{ \cU:=u - i\LR D^{-1}\dot u,\quad  \cN:=n -
iD^{-1}\dot n/\al,} where $\LR x=(1+x^2)^{1/2}$, then $u=\re
\cU=(\cU+\bar \cU)/2$, $n=\re \cN=(\cN+\bar \cN)/2$ and the
equations for $(\cU,\cN)$ are \EQ{\label{eq:KGZ2}
 \CAS{ (i\p_t+\LR D)\cU =\LR D^{-1}(\cN\cU/4+\bar{\cN}\cU/4+
 \cN\bar{\cU}/4+\bar{\cN}\bar{\cU}/4),\\
   (i\p_t+\al D) \cN = \al D(\cU\bar{\cU}/4+\bar{\cU}\cU/4+\cU^2/4+\bar{\cU}^2/4).}}

Now we introduce some notations. We use $K(t),W_\alpha(t)$ to denote
the Klein-Gordon and the wave propagators:
\[K(t)\phi=\ft^{-1}e^{it\LR\xi}\hat{\phi},\quad W_\alpha(t)\phi=\ft^{-1}e^{i\alpha t|\xi|}\hat{\phi}, \pq \hat\phi=\F\phi.\]
Let $\eta_0: \R\rightarrow [0, 1]$ denote a radial smooth function
supported in $\{|\xi|\leq 2\}$ and equal to $1$ in $\{|\xi|\leq
1\}$. For $k\in \Z$ let
$\chi_k(\xi)=\eta_0(\xi/2^k)-\eta_0(\xi/2^{k-1})$ and $\chi_{\leq
k}(\xi)=\eta_0(\xi/2^k)$. For $k\in \Z$ let $P_k, P_{\leq k}$ denote
the operators on $L^2(\R^3)$ defined by
$\widehat{P_ku}(\xi)=\chi_k(|\xi|)\widehat{u}(\xi),\widehat{P_{\leq
k}u}(\xi)=\chi_{\leq k}(|\xi|)\widehat{u}(\xi)$.

For a quadratic term $uv$, we use $(uv)_{LH}$, $(uv)_{HH}$,
$(uv)_{HL}$ to denote the three different interactions
\[
 (uv)_{LH}=\sum_{k\in\Z}P_{\leq k-k_\alpha}uP_kv,\quad (uv)_{HL}=(vu)_{LH},\quad (uv)_{HH}=\sum_{\substack{|k_1-k_2|< k_\alpha \\ k_1,k_2\in\Z}}P_{k_1}uP_{k_2}v,\]
 where $k_\alpha$ is a large number which is determined later, depending on $\alpha$.
It is obvious that we have
\EQ{
 uv\pt=(uv)_{HH}+(uv)_{LH}+(uv)_{HL},}
and they are all radial if $u,v$ are both radial.
Moreover, for any such index $*=HH,HL,LH$, we denote the bilinear symbol (multiplier) by
\EQ{
 \F(uv)_* = \int \cP_*\hat u(\x-\y)\hat v(\y)d\y,}
and finite sum of those bilinear operators are denoted by the sum of indices:
\EQ{
 (uv)_{*_1+*_2+\cdots}=(uv)_{*_1}+(uv)_{*_2}+\cdots.}

From Duhamel's formula and taking a Fourier transform, we get that
the first equation of \eqref{eq:KGZ2} is equivalent to
\begin{align*}
\hat{\cU} &=e^{it\LR\xi}\hat{\cU_0}-i\LR\xi^{-1}\int_0^t
e^{i(t-s)\LR\xi}\ft(nu)_{HL} ds
 -i\LR\xi^{-1}\int_0^t e^{i(t-s)\LR\xi}\ft(nu)_{HH+LH}ds
\end{align*}
Especially, for the second term, we have
\begin{align*}
&-i\LR\xi^{-1}\int_0^t \int e^{i(t-s)\LR\xi}\cP_{HL} \hat{n}(s,\xi-\eta)\hat{u}(s,\eta)d\eta
ds
\\
=&-\frac{i}{4}\LR\xi^{-1}e^{it|\xi|^2}\int_0^t \int \cP_{HL} e^{is\om_1}[e^{-i\alpha s|\xi-\eta|}\hat{\cN}(s,\xi-\eta)][e^{-is\LR\eta}\hat{\cU}(s,\eta)]d\eta ds
\\
&-\frac{i}{4}\LR\xi^{-1}e^{it|\xi|^2}\int_0^t \int \cP_{HL} e^{is\om_2}[e^{i\alpha s|\xi-\eta|}\hat{\bar\cN}(s,\xi-\eta)][e^{-is\LR\eta}\hat{\cU}(s,\eta)]d\eta ds
\\
&-\frac{i}{4}\LR\xi^{-1}e^{it|\xi|^2}\int_0^t \int \cP_{HL} e^{is\om_3}[e^{-i\alpha s|\xi-\eta|}\hat{\cN}(s,\xi-\eta)][e^{is\LR\eta}\hat{\bar\cU}(s,\eta)]d\eta ds
\\
&-\frac{i}{4}\LR\xi^{-1}e^{it|\xi|^2}\int_0^t \int \cP_{HL} e^{is\om_4}[e^{i\alpha s|\xi-\eta|}\hat{\bar\cN}(s,\xi-\eta)][e^{is\LR\eta}\hat{\bar\cU}(s,\eta)]d\eta ds,
\end{align*}
%\begin{align*}
%&-i\LR\xi^{-1}\int_0^t \int e^{i(t-s)\LR\xi}\cP_{HL} \hat{n}(s,\xi-\eta)\hat{u}(s,\eta)d\eta
%ds
%\\
%=&-\frac{i}{4}\LR\xi^{-1}e^{it|\xi|^2}\int_0^t \int \cP_{HL} \left(e^{is\om_1}\tilde{\cN}(s,\xi-\eta)\tilde{\cU}(s,\eta)
%+e^{is\om_2}\bar{\tilde{\cN}}(s,\xi-\eta)\tilde{\cU}(s,\eta)\right.\\
%&\left.\ \ \ \ \ \ \ \ \ \ \ \ \ \ \ \ \ \ \ \ \ \ \ \ +e^{is\om_3}\tilde{\cN}(s,\xi-\eta)\bar{\tilde{\cU}}(s,\eta)
%+e^{is\om_4}\bar{\tilde{\cN}}(s,\xi-\eta)\bar{\tilde{\cU}}(s,\eta)\right)d\eta ds
%\end{align*}
%where
%\begin{align*}
%\tilde{\cN}(t,\x)=e^{-i\alpha t|\xi|}\hat{\cN}(t,\xi),\quad\tilde{\cU}(t,\x)=e^{-it\LR\x}\hat{\cU}(t,\x),
%\end{align*}
where
\begin{align*}
& \om_1=-\LR\x+\al|\x-\y|+\LR\y,
& \om_2=-\LR\x-\al|\x-\y|+\LR\y,\\
& \om_3=-\LR\x+\al|\x-\y|-\LR\y,
& \om_4=-\LR\x-\al|\x-\y|-\LR\y.
\end{align*}

 It is obvious that
$\om_2$ and $\om_4$ will not vanish in the support of $\cP_{HL}$: $|\xi|\sim |\xi-\eta|\gg |\eta|$. For example, if we choose $k_\alpha \ge 5$, then
\begin{align*}
|\om_2|, |\om_4|\sim_\alpha \LR\xi.
\end{align*}
Therefore, there is no resonance in these cases.

In contrast, $\om_1$ and $\om_3$ have more trouble since they
vanish when $|\eta|=0$ and $|\x|=c_\alpha:=2\alpha/|\alpha^2-1|$ in
the support of $\cP_{HL}$. Therefore, we need further to
distinguish $(uv)_{HL} $ between resonant and non-resonant frequency
parts as follows \EQ{
 (uv)_{\al L}=\sum_{\substack{|2^k-c_\alpha|\le \delta_\alpha,\\ k\in\Z}}P_kuP_{\leq k-k_\alpha}v,\quad
 (uv)_{X L}=\sum_{\substack{|2^k-c_\alpha|>  \delta_\alpha,\\ k\in\Z}}P_kuP_{\leq k-k_\alpha}v,}
and similarly denote $(uv)_{L\al}$, $(uv)_{LX}$. Then we use normal
form only for non-resonant parts. We give the estimates of $\om_1$
and $\om_3$ precisely in the following lemma, similar to the
estimates in \cite{MN1}:
\begin{lem} Let $1\ne \alpha>0$, then there exist $c_\alpha$, $\delta_\alpha$ and $k_\alpha$   such that
in the support of $\cP_{XL}$,
\[|\om_1|\sim_\alpha |\xi|,\quad |\om_3|\sim_\alpha \LR{\xi}.\]
\end{lem}

\begin{proof}
We will use the simple fact $\LR\eta-1=\frac{|\eta|^2}{\LR\eta+1}
\le |\eta|$.

(1) We consider the case $0<\alpha<1$.

  For $\om_1$, by solving
\[\LR\x-1=\al|\x|,\]
we can get the resonant frequency
\[c_\alpha=\frac{2\alpha}{1-\alpha^2}.\]
 Now we estimate the function $f(r):=\alpha r-\LR r+1$. Since $f'(r)=\alpha-r/\LR r$ and
$f''(r)=-1/\LR r^3$, $f$ is convex and has only maximum at
\[r_0=\frac{\alpha}{\sqrt{1-\alpha^2}}\in (0, c_\alpha).\]
There exists $\theta\in(0,1)$ such that $c_\alpha(1-\theta)\in (r_0,
c_\alpha)$. Let $\delta_\alpha=\theta c_\alpha$, then we can find a
number $\rho=\rho(\alpha,\delta_\alpha)$ such that
\[|f(r)|\geq \rho r \ \text { for } \ r\in [0, c_\alpha-\delta_\alpha)\cup(c_\alpha+\delta_\alpha,\I).\]
Choosing $k_{\alpha} \ge |\log_2\rho|+5$, we have
\[|\om_1|\sim_\alpha |\xi|\]
in the support of $\cP_{XL}$.

Now we consider $\om_3$. Choosing  $k_{\alpha} \ge |\log_2(1-\alpha)|+5$, we have 
$(1-\alpha)|\xi|\gg |\eta|$. Since
\[
|\om_3|
 \ge |\x|-\al|\x-\y|+1
 \ge (1-\al)|\x|-|\y|+1
 \ge c |\x|+1,\]
 we have
\[|\om_3|\sim_\alpha \LR\xi\]
in the support of $\cP_{HL}$.

(2) We consider the case $\alpha>1$.

For $\om_1$, by choosing $k_{\alpha}\ge|\log_2(\alpha-1)|+5$, we
have $|\xi|\gg |\eta|$ and $(\alpha-1)|\xi|\gg |\eta|$, and hence
\begin{align*}
|\om_1|=|(-\LR\x+1)+\al|\x-\y|+(\LR\y-1)|\sim_\alpha |\xi|
\end{align*}
 in the support of $\cP_{HL}$.

For $\om_3$, by solving
\[\LR\x+1=\al|\x|,\]
we can get the resonant frequency
\[c_\alpha=\frac{2\alpha}{\alpha^2-1}.\]
For the function $g(r):=\alpha r-\LR
r-1$,
since $g'(r)=\alpha-r/\LR r>0$, $g''(r)=-1/\LR r^3<0$ and the
asymptotic line is $y(r)=(\alpha-1)r-1$ when $r\rightarrow \I$,
$|g(r)|$ and  the line $h(r):=(\alpha-1)r/2$ have two crossing
points $r_{c1}$ and $r_{c2}$. Let $\delta_\alpha$ such that
\begin{align*}
\delta_\alpha=\max\{|c_\alpha-r_{c1}|,\ |c_\alpha- r_{c2}|\},
\end{align*}
we have
\[|g(r)|\geq \frac{\alpha-1}{2} r \ \text { for } \ r\in [0,c_\alpha-\delta_\alpha)\cup(c_\alpha+\delta_\alpha,\I).\]
Choosing $k_{\alpha} \ge |\log_2(\alpha-1)|+5$, 
 and noting that
$|\om_3|\sim 1$ for $1\gg|\x|\gg|\y|$,  we have
\[|\om_3|\sim_\alpha\LR\xi\]
in the support of $\cP_{XL}$.
\end{proof}

By the lemma above, we gain $|\x|^{-1}$ for high frequencies
($|\x|>1$) in all the cases, and lose $|\x|^{-1}$ for low
frequencies ($|\x|<1$) in the case  $\om_1$. In general, the lower
frequencies can be more problematic in the scattering problems, but
it will turn out that we can absorb $|\x|^{-1}$ by the Sobolev
embedding.

By similar analysis, corresponding to the four nonlinear terms of the second equation of \eqref{eq:KGZ2}, the resonance functions are
\begin{align*}
& \tilde{\om}_1=-\al|\x|+\LR{\x-\y}-\LR\y,
& \tilde{\om}_2=-\al|\x|-\LR{\x-\y}+\LR\y,\\
& \tilde{\om}_3=-\al|\x|+\LR{\x-\y}+\LR\y,
& \tilde{\om}_4=-\al|\x|-\LR{\x-\y}-\LR\y.
\end{align*}
It is easy to check that $|\tilde{\om}_j|$ behaves the same as $|\om_j|$ for $j=1,2,3,4$. Indeed, $\om_j$ and $\ti\om_j$ are in the dual relation with the correspondence $\x \mapsto \y-\x$.

In order to simplify the presentation, we assume that $\alpha<1$\footnote{This is the physical case in plasma} and the nonlinear terms in the first and second equation of \eqref{eq:KGZ2} are $\cN\cU$ and $\cU\bar\cU$ respectively. For other cases, the proof is almost the same.
Then we get that
the first equation of \eqref{eq:KGZ2} is equivalent to
\begin{align*}
\hat{\cU} &=e^{it\LR\xi}\hat{\cU_0}-i\LR\xi^{-1}\int_0^t
e^{i(t-s)\LR\xi}\ft(\cN\cU)_{XL} ds
 -i\LR\xi^{-1}\int_0^t e^{i(t-s)\LR\xi}\ft(\cN\cU)_{HH+LH+\alpha L}ds\\
&:=I+II+III.
\end{align*}
Using the equation
\eqref{eq:KGZ2} again, we get that
\begin{align}
\partial_t(e^{-it\LR\xi}\hat{\cU})=&-ie^{-it\LR\xi}\LR\xi^{-1}(\hat{\cN}*\hat{\cU})(\xi),\\
\partial_t(e^{-i\alpha t|\xi|}\hat{\cN})=&-ie^{-i\alpha t|\xi|}\alpha |\xi|(\hat{\cU}*\hat{\bar
\cU})(\xi).
\end{align}
Thus we have
\begin{align*}
II=&-i\LR\xi^{-1}\int_0^t \int e^{i(t-s)\LR\xi}\cP_{XL} \hat{\cN}(s,\xi-\eta)\hat{\cU}(s,\eta)d\eta
ds\\
=&-i\LR\xi^{-1}e^{it\LR\xi}\int_0^t \int \cP_{XL} e^{is\om}[e^{-i\alpha s|\xi-\eta|}\hat{\cN}(s,\xi-\eta)][e^{-is\LR\eta}\hat{\cU}(s,\eta)]d\eta ds,
\end{align*}
where the resonance function
 \[\om:=-\LR\x+\al|\x-\y|+\LR\y.\]
From integration by parts, we get
\begin{align*}
II=&-\LR\xi^{-1}e^{it\LR\xi}\int_0^t \int \cP_{XL}\om^{-1}\partial_s(e^{is\om })e^{-i\alpha
s|\xi-\eta|}\hat{\cN}(s,\xi-\eta)e^{-is\LR\eta}\hat{\cU}(s,\eta)d\eta ds\\
=&-\LR\xi^{-1}\int\cP_{XL}\om^{-1}[\hat{\cN}(t,\xi-\eta)\hat{\cU}(t,\eta)-e^{it\LR\xi}\hat{\cN}(0,\xi-\eta)\hat{\cU}(0,\eta)]d\eta\\
&-i\alpha\LR\xi^{-1}\int_0^t \int \cP_{XL}\om^{-1}e^{i(t-s)\LR\xi
}|\xi-\eta|\widehat{|\cU|^2}(s,\xi-\eta)\hat{\cU}(s,\eta)d\eta
ds\\
&-i\LR\xi^{-1}\int_0^t \int\cP_{XL}\om^{-1} e^{i(t-s)\LR\xi }\hat{\cN}(s,\xi-\eta)\LR\eta^{-1}(\hat{\cN}*\hat{\cU})(s,\eta)d\eta ds.
\end{align*}
We introduce a bilinear Fourier multiplier in the form
\EQ{
 \Om(f,g)=\F^{-1}\int\cP_{XL}\om^{-1}\hat f(\x-\y)\hat g(\y)d\y.}
Then we have
\begin{align*}
II=&-\LR\xi^{-1}\ft\Om(\cN,\cU)(t)+\LR\xi^{-1}e^{it\LR\xi}\ft\Om(\cN,\cU)(0)\\
&-i\alpha\LR\xi^{-1}\int_0^te^{i(t-s)\LR\xi}\ft
\Om(D|\cU|^2,\cU)(s)ds\\
&-i\LR\xi^{-1}\int_0^te^{i(t-s)\LR\xi}\ft \Om(\cN,\LR
D^{-1}(\cN\cU))(s)ds.
\end{align*}
Thus we obtain
\begin{equation}\label{int-U}
\begin{split}
\cU=&K(t)\cU_0+K(t)\LR D^{-1}\Om(\cN,\cU)(0)-\LR D^{-1}\Om(\cN,\cU)(t)\\
&-i\alpha\LR D^{-1}\int_0^tK(t-s)\Om(D|\cU|^2,\cU)(s)ds\\
&-i\LR D^{-1}\int_0^tK(t-s)\Om(\cN,\LR D^{-1}(\cN\cU))(s)ds\\
&-i\LR D^{-1}\int_0^tK(t-s)(\cN\cU)_{HH+LH+\alpha L}ds.
\end{split}
\end{equation}

For the second equation in \eqref{eq:KGZ2}, similarly, we can apply
the normal form reduction for the high-low and low-high interaction, and then get
that it is equivalent to
\begin{equation}\label{int-N}
\begin{split}
\cN=&W_\alpha(t)N_0+\alpha W_\alpha(t)D\tilde\Om(\cU,\cU)(0)-\alpha D\tilde\Om(\cU,\cU)(t)\\
&-i\alpha\int_0^tW_\alpha(t-s)D(\cU\bar
\cU)_{HH+\alpha L+L\alpha}ds\\
&-i\alpha\int_0^tW_\alpha(t-s)(D\tilde\Om(\LR D^{-1}(\cN\cU),\cU)+D\tilde\Om(\cU,\LR D^{-1}(\cN\cU)))(s)ds,
\end{split}
\end{equation}
where $\tilde \Om$ is a bilinear Fourier multiplier in the form \EQ{
 \tilde\Om(f,g)=\F^{-1}\int \cP_{XL+LX}\frac{\hat f(\x-\y)\hat{\bar g}(\y)}{\LR{\xi-\eta}-\LR\eta-\alpha|\xi|}d\y.}

\section{Strichartz estimates and nonlinear estimates}

In this section, we introduce the Strichartz norm we need. Because
of the quadratic term, our space relies heavily on the radial
symmetry.
For $\cU$ and $\cN$, we use the radial-improved Strichartz
norms \EQ{ \label{Strz norms}
 \cU \in X|Y,\quad
 \cN \in L^\I_tL^2_x \cap L^2_t\dot B^{-1/4-\e}_{q(-\e),2},}
for fixed $0<\e\ll 1$, where $\|\cU\|_{X|Y}:=\|P_{<0}\cU\|_{X}+\|P_{\geq0}\cU\|_{Y}$ and
\begin{align*}
X=L^\I_tL^2_x \cap L_t^2\dot B^{1/4+\e}_{q(\e),2}, \quad
Y=L^\I_tH^1_x \cap L_t^2 B^{2/3}_{q(\e),2}, \quad
\frac{1}{q(\e)}=\frac{1}{4}+\frac{\e}{3}.
\end{align*}
By the Sobolev embedding,
\begin{align*}
\dot H_x^1 = \dot B^1_{q(3/4),2} \subset \dot B^{1/4+\e}_{q(\e),2} \subset \dot B^{1/4-\e}_{q(-\e),2}\subset L_x^6,\\
H_x^{\frac{17}{12}-\e} \subset B^{\frac{2}{3}}_{q(\e),2} \subset  B^{\frac{2}{3}-2\e}_{q(-\e),2}\subset  B^{\frac{5}{12}-\e}_{6,2}\subset L_x^6.
\end{align*}
From now on, the third exponent of the Besov space will be fixed to
$2$ and so omitted. The condition $0<\e\ll 1$ ensures that \EQ{
 \frac{10}{3}<q(\e)<4<q(-\e)<\I,}
such that the norms in \eqref{Strz norms} are Strichartz-admissible
for radial solutions. The Strichartz estimates that we will use are
given in the following lemma.

\begin{lem}\label{lem:radstri} Assume that $\phi(x)$, $f(t,x)$ are spatially radially
symmetric in $\R^3$. Then

(a) Assume $(q,r),(\tilde q,\tilde r)\in [2,\infty]^2$ both satisfy
the Schr\"{o}dinger-admissible condition:
\[\frac{2}{q}+\frac{5}{r}<\frac{5}{2} \mbox{ or } (q,r)=(\infty,2)\]
and $\tilde q>2$. Let
\[\beta(q,r)=
\begin{cases}
\frac{3}{2}-\frac{3}{r}-\frac{1}{q},\quad  \frac{1}{q}+\frac{2}{r}<1
\mbox{ or } (q,r)=(\infty,2);\\
\frac{1}{r}+\frac{1}{q}-\frac{1}{2}, \quad \frac{1}{q}+\frac{2}{r}>1
\mbox{ and }
\frac{2}{q}+\frac{5}{r}<\frac{5}{2};\\
(\frac{1}{2}-\frac{1}{r})+, \quad \frac{1}{q}+\frac{2}{r}=1.
\end{cases}\]
where we used the notation $a+$ to denote $a+\e$ for arbitrary fixed
$\e>0$. Then
\begin{align}
\norm{K(t)P_{\geq0}\phi}_{L_t^q \dot{B}_{r,2}^{-\beta(q,r)}}\les&
\norm{\phi}_{L_x^2},\\
\normo{\int_0^tK(t-s)P_{\geq0}f(s)ds}_{L_t^q
\dot{B}_{r,2}^{-\beta(q,r)}}\les& \norm{P_{\geq0}f}_{L_t^{\tilde
q'}\dot{B}_{\tilde r',2}^{\beta(\tilde q,\tilde r)}},\\
\norm{K(t)P_{<0}\phi}_{L_t^q
\dot{B}_{r,2}^{\frac{2}{q}+\frac{3}{r}-\frac{3}{2}}}\les&
\norm{\phi}_{L_x^2},\\
\normo{\int_0^tK(t-s)P_{<0}f(s)ds}_{L_t^q
\dot{B}_{r,2}^{\frac{2}{q}+\frac{3}{r}-\frac{3}{2}}}\les&
\norm{P_{<0}f}_{L_t^{\tilde q'}\dot{B}_{\tilde
r',2}^{\frac{3}{2}-\frac{3}{\tilde r}-\frac{2}{\tilde q}}}.
\end{align}

(b) if $(q,r),(\tilde q,\tilde r)\in [2,\infty]^2$ both satisfy the
wave-admissible condition:
\[\frac{1}{q}+\frac{2}{r}<1 \mbox{ or } (q,r)=(\infty,2)\] and $\tilde q>2$, then
\begin{align}
\norm{W_\alpha(t)\phi}_{L_t^q\dot{B}_{r,2}^{\frac{1}{q}+\frac{3}{r}-\frac{3}{2}}}\les&
\norm{\phi}_{L_x^2},\\
\normo{\int_0^tW_\alpha(t-s)f(s)ds}_{L_t^q\dot{B}_{r,2}^{\frac{1}{q}+\frac{3}{r}-\frac{3}{2}}}\les&
\norm{f}_{L_t^{\tilde q'}\dot{B}_{\tilde
r',2}^{\frac{3}{2}-\frac{3}{\tilde r}-\frac{1}{\tilde q}}}.
\end{align}
\end{lem}
\begin{proof}
The proof of (b) can be found in \cite{GuoWang}, and the previous references therein. Using their idea, we give a rough proof for both (a) and (b).
By Riesz-Thorin interpolation and the classical Strichartz estimates, it suffices to prove the lemma for $(q,r)=(2,r)$.
Consider a free solution on $\R^3$ with $|\x|\sim 2^k$ frequency in the form
\EQ{
 u_k(t,x)=e^{it\om(D)}P_k\phi(x),}
where $\phi\in L^2_x$ is radial, and $\om(|\xi|)$ is the dispersion
function. Computing it in polar coordinate, we have
\EQ{\label{eq:FB}
 u_k(t,x)=\frac{4\pi}{|x|}\int_{\ro\sim 1}e^{it\om(2^k\ro)}\chi_0(\ro)2^{2k}\ro\hat{\phi}(2^k\ro)\sin(2^k|x|\ro)d\ro.}
Hence if for some $j$, we have an estimate of the form \EQ{
\label{L2Linf bd}
 \|\chi_j(|x|)\int_\R e^{it\om(2^k\ro)+i2^kx\ro}\chi_0(\ro)f(\ro)d\ro\|_{L^2_tL^r_x(\R^2)} \lec 2^{\al j+\beta k}\|f\|_{L_x^2(\R)},}
with some  $\al,\beta$ and $r\ge 2$, then we get \EQ{
 \|\chi_j(|x|)u_k(t,x)\|_{L^2_t L^r_x(\R^{1+3})} \lec 2^{(\al-1+2/r) j+(1/2+\beta) k}\|\phi\|_{L^2_x(\R^3)}.}
Let $Tf$ be the inside of the norm on the left of \eqref{L2Linf bd}.
Then we have \EQ{
 \pt T^*F=\iint e^{-is\om(2^k\x)-i2^ky\x}\chi_0(\x)\chi_j(|y|)F(s,y)dyds,
 \pr TT^*F=\iint e^{i(t-s)\om(2^k\x)+i2^k(x-y)\x}\chi_0^2(\x)d\x \cdot\chi_j(|x|)\chi_j(|y|)F(s,y)dyds,}
and so \EQ{
 |TT^*F| \le |(e^{it\om(2^kD)}\F^{-1}\chi_0^2)(2^kx)\chi_{\leq j+1}(|x|)|*|F|.}
Let $K(t,x)=(e^{it\om(2^kD)}\F^{-1}\chi_0^2)(2^kx)$. Then
\eqref{L2Linf bd} will follow from \EQ{
 \|K(t,x)\|_{L^1_{t\in\R}L^{r/2}_{|x|<2^j}} \lec 2^{2\al j+2\beta k}.}

(a) In the Klein-Gordon case $\om(\ro)=\LR\ro$, \EQ{
 K(t,x)=\int e^{it\LR{2^k\ro}+i2^kx\ro}\chi_0^2(\ro)d\ro.}
Simple computation shows that $\om'(\rho)=\rho\LR{\rho}^{-1}$,
$\om''(\rho)=\LR{\rho}^{-3}$. For $r=2$, we use the local smoothing
estimates. Indeed, using the Plancherel's identity in $t$ and
Cauchy-Schwartz inequality in $x$, we get
\begin{align*}
\norm{Tf}_{L_t^2L_x^2}\les 2^{j/2}2^{-k/2}\LR{2^{k}}^{1/2}\norm{f}_2,
\end{align*}
and hence \EQ{\label{(2,2)}
 \|e^{it\LR{D}}P_k\phi\|_{L^2_t L^2_{|x|\sim 2^j}} \les 2^{j/2}2^{-(k\wedge
 0)/2}\norm{\phi}_{L_x^2}}
where we used the notation $a\vee b=\max(a,b),\ a\wedge
b=\min(a,b)$.

Let $\psi(\ro)=t\LR {2^k\ro}+2^k\ro x$ . Then
$\psi'(\rho)=t2^{2k}\rho\LR{2^k\rho}^{-1}+2^kx$. Thus if
$|t|2^{k\wedge 0}\gg 2^j$, then $|\psi'(\rho)|\ges
|t|2^{2k}\LR{2^k}^{-1}$, using integration by parts twice we get
\[|K(t,x)|\leq \int \big|\partial_\rho[\psi'(\rho)^{-1}\partial_\rho
(\chi_0^2(\rho)\psi'(\rho)^{-1})]\big| d\rho\les
|t|^{-2}2^{-4k}\LR{2^k}^2.\] Combining with the trivial bound
$|K|\les 1$, we get that for $k\geq0$ and $j\geq -k$,
 \EQ{
 \|K\|_{L^1_t L^\I_{|x|<2^{j}}} \lec \int_{|t|<2^{j+2}}dt+\int_{|t|>2^{j+2}}2^{-2k}t^{-2}dt \lesssim 2^j.}
Hence,  for $k\geq0$ and $j\geq -k$,
\begin{align}\label{(2,I)-W-high}
 \|e^{it\LR D}P_k\phi\|_{L^2_tL^\I_{|x|\sim2^j}} \lec
2^{(k-j)/2}\|\phi\|_{L^2_x}.
\end{align}
Interpolating \eqref{(2,2)}  with \eqref{(2,I)-W-high} and classical
Strichartz estimate
 \EQ{ \|e^{it\LR{D}}P_k\phi\|_{L^2_tL^\I_{x\in
\R^3}} \les 2^{(k\wedge 0)/2}2^{k\vee 0}\|\phi\|_{L^2_x},} we can
get the homogeneous estimates in part (a) for wave-admissible pairs.

Now we use the stationary phase method to get an improvement due to the
non vanishing second derivative. Indeed \EQ{ |\psi''(\ro)|=
|\frac{2^{2k}t}{\LR {2^k\ro}^{3}}|\gtrsim |t|2^{-(k\vee
0)}2^{2(k\wedge 0)}}in the support of $\chi_0$. Hence by the
stationary phase method\EQ{\label{K-high}
 |K(t,x)|\les |t|^{-1/2}2^{(k\vee 0)/2}2^{-(k\wedge 0)}.}
Thus eventually we have
\[|K(t,x)\chi_j(x)|\les  |t|^{-1/2}2^{(k\vee 0)/2}2^{-(k\wedge 0)}1_{\{|t|2^{k\wedge 0}\les 2^j\}}+|t|^{-2}2^{-4k}\LR{2^k}^21_{\{|t|2^{k\wedge 0}\gg 2^j\}}.\]
Therefore,
\[\norm{K}_{L_t^1L^\infty_{|x|\les 2^j}}\les
2^{j/2}2^{(k\vee 0)/2}2^{-3(k\wedge 0)/2}+2^{-j}2^{k\wedge
0}2^{-4k}\LR{2^{2k}},\] and then for $j\geq -\frac{5}{3}(k\vee
0)-(k\wedge 0)$ we have
\begin{align}\label{(2,I)-K}
 \|e^{it\LR D}P_k\phi\|_{L^2_tL^\I_{|x|\sim2^j}} \lec
2^{-3j/4+3(k\vee 0)/4-(k\wedge 0)/4}\|\phi\|_{L^2_x}.
\end{align}
In particular, by interpolation between \eqref{(2,I)-K} and
\eqref{(2,I)-W-high}, we get for $k\geq 0,j\geq -k$,
\begin{align}\label{(2,I)}
 \|e^{it\LR D}P_k\phi\|_{L^2_tL^\I_{|x|\sim2^j}} \lec
2^{\theta(k-j)}\|\phi\|_{L^2_x}, \, \frac{1}{2}\leq \theta \leq
\frac{3}{4}.
\end{align}

Interpolating \eqref{(2,2)} with \eqref{(2,I)} and classical
Strichartz estimates, we get that for $k<0$, if $r>10/3$, then
\begin{align*}
 \|e^{it\LR{D}}P_k\phi\|_{L^2_tL_x^r} &\lesssim
 (\sum_{j\leq -k}2^{\frac{j-k}{r}+\frac{k}{2}(1-\frac{2}{r})}
 +\sum_{j>-k}2^{\frac{j-k}{r}-\frac{3j+k}{4}(1-\frac{2}{r})})\|\phi\|_{L^2_x}\les 2^{k(\frac{1}{2}-\frac{3}{r})}\|\phi\|_{L^2_x};
\end{align*}
for $k\geq 0$, if $\frac{10}{3}<r<4$, then
\begin{align*}
 \|e^{it\LR{D}}P_k\phi\|_{L^2_tL_x^r} \les&
 (\sum_{j\leq -k}2^{\frac{j}{r}+k(1-\frac{2}{r})}+\sum_{j\leq k}2^{\frac{j}{r}+\frac{1}{2}(k-j)(1-\frac{2}{r})}+
 \sum_{j\geq
 k}2^{\frac{j}{r}+\frac{3}{4}(k-j)(1-\frac{2}{r})})\|\phi\|_{L^2_x}\\
 \les& 2^{\frac{k}{r}}\|\phi\|_{L^2_x};
\end{align*}
for $k\geq 0$, if $r=4$, then
\begin{align*}
 \|e^{it\LR{D}}P_k\phi\|_{L^2_tL_x^r} \les&
 (\sum_{j\leq -k}2^{\frac{j}{r}+k(1-\frac{2}{r})}+\sum_{-k\leq j\leq
 k}2^{k/4}+
 \sum_{j>k}2^{\frac{j}{r}+\frac{3}{4}(k-j)(1-\frac{2}{r})})\|\phi\|_{L^2_x}\\
 \les &\jb{k}2^{\frac{k}{4}}\|\phi\|_{L^2_x};
\end{align*}
for $k\geq 0$, if $r>4$, then
\begin{align*}
 \|e^{it\LR{D}}P_k\phi\|_{L^2_tL_x^r} \les&
 (\sum_{j\leq -k}2^{\frac{j}{r}+k(1-\frac{2}{r})}+
 \sum_{j>-k}2^{\frac{j}{r}+\frac{1}{2}(k-j)(1-\frac{2}{r})})\|\phi\|_{L^2_x}
 \les 2^{k(1-\frac{3}{r})}\|\phi\|_{L^2_x}.
\end{align*}
Therefore, by interpolation the homogeneous estimates for
Schr\"{o}dinger-admissible pairs in part (a) is proved.

The inhomogeneous linear estimates follow from the duality argument
and the Christ-Kiselev lemma, similar to \cite{GuoWang}.

(b)
In the wave case $\om(\ro)=|\ro|$, we have
\EQ{
 \|\F^{-1}(\chi_0^2)(x\pm t)\|_{L^1_t L^{r/2}_{|x|<2^j}} \sim 2^{j},}
hence $\al=1/2$ (independent of $r$). Thus we obtain
\EQ{
 \||x|^{1/2-2/r}e^{it|D|}\chi_0(D)\phi\|_{L^2_tL^r_{|x|\sim2^j}} \lec \|\phi\|_{L^2_x}.}
In particular, we have
\EQ{
 \|e^{it|D|}\chi_0(D)\phi\|_{L^2_tL^r_x} \lec \|\phi\|_{L^2_x}\ (\forall r>4).}
 By scaling,
\EQ{
 \|e^{it|D|}P_k\phi\|_{L^2_tL^r_x} \lec (2^k)^{\frac{3}{2}-\frac{1}{2}-\frac{3}{r}}\|\phi\|_{L^2_x}\ (\forall r>4).}
This yields the radial improvement of the wave Strichartz in 3D.
\end{proof}

\begin{rem}
The generalized Strichartz estimates for Klein-Gordon equation was
also studied by Cho-Lee \cite{CL} which also addresses the
non-radial versions. Our proof is different from theirs, and the
idea is from \cite{GuoWang}. Our results give better bound on the
regularity, but the range of $(q,r)$ is the same except some
endpoints. More precisely, they prove that the borderline case $2/q+5/r=5/2$ is also admissible except for the endpoint $(q,r)=(2,10/3)$.
The borderline case for the Schr\"{o}dinger equation was partially proved in \cite{GuoWang}, which was extended except for the endpoint by \cite{K}.
The borderline case for the wave equation is prohibited except for the trivial energy norm.

The regularity in our estimates is optimal for all $(q,r)$ in the admissible range. Indeed,  there exists radial $L^2$
function $\phi\ne 0$ such that
\begin{align}\label{eq:sharpKG}
\|e^{it\LR{D}}P_k\phi\|_{L^q_tL_x^r} \ges C(q,r,k) \|\phi\|_{L^2_x},
\end{align}
where $C(q,r,k)=\LR{k}^{1/q}2^{(1/2-1/r)k}$ for $(q,r)$ satisfying $1/q+2/r=1$, and $C(q,r,k)=2^{\beta(q,r)k}$ for all other $(q,r)$ in the admissible range. By \eqref{eq:FB}, \eqref{eq:sharpKG}
is equivalent to the existence of $f$ such that
\begin{align}\label{eq:sharpKG2}
\|s^{\frac{2}{r}-1}\int_{\rho\sim
1}e^{it\LR{2^k\rho}}\chi_0(\rho)f(\rho)\sin(2^k
s\rho)d\rho\|_{L^q_tL_{s>0}^r} \ges C(q,r,k) 2^{-k/2}\|f\|_{L^2}.
\end{align}
Take $f=1_{[0,10]}(\rho)$, then we have
\begin{align*}
I:=&\int_{\rho\sim 1}e^{it\LR{2^k\rho}}\chi_0(\rho)f(\rho)\sin(2^k
s\rho)d\rho\\
=&\frac{1}{2i}\int_{\rho\sim
1}e^{it\LR{2^k\rho}}\chi_0(\rho)(e^{i2^k
s\rho}-e^{-i2^ks\rho})d\rho:=I_1-I_2.
\end{align*}
In the region $E=\{2^{-k}\ll |t| \ll 2^k$, $|t-s|\ll 2^{-k}\}$,
using integration by parts we get $|I_1|\ll 1$; on the other hand
\[|I_2|\sim |\int_{\rho\sim 1}e^{\frac{it}{\LR{2^k\rho}+2^k\rho}}e^{i2^k(t-s)\rho}\chi_0(\rho)d\rho|\sim 1.\]
Thus $|I|\sim 1$ on $E$. Hence,
\[L.H.S. \mbox{ of } \eqref{eq:sharpKG2}\ges 2^{-k/r}(\int_{2^{-k}\ll |t| \ll 2^k} |t|^{\frac{2q}{r}-q}dt)^{1/q}
\ges C(q,r,k)2^{-k/2}\norm{f}_2,\] and \eqref{eq:sharpKG} is proved.
\end{rem}

We will apply this lemma to the integral equations. Then in
order to close the argument, we need to do some nonlinear estimates.
\subsection{Bilinear terms} The above Strichartz norms neatly fit in
the bilinear terms on the right, which are partially resonant.
Indeed we have
\begin{lem}\label{lem:bi1} (1) For any $\cN$ and $\cU$, the following estimates hold
\begin{align*}
 \|\LR D^{-1}(\cN\cU)_{LH}\|_{L^1_tH^1_x} \les& \|\cN\|_{L^2_t\dot B^{-1/4-\e}_{q(-\e)}}
 \|\cU\|_{L_t^2\dot B^{1/4+\e}_{q(\e)}},\\
 \|\LR D^{-1}(\cN\cU)_{HH}\|_{L^1_tH^1_x} \les& \|\cN\|_{L^2_t\dot B^{-1/4-\e}_{q(-\e)}}
  \|\cU\|_{L_t^2\dot B^{1/4+\e}_{q(\e)}}.
 \end{align*}
 If $0\leq \theta\leq 1$, $\frac{1}{\tilde
q}=\frac{1}{2}-\frac{\theta}{2}$, $\frac{1}{\tilde
r}=\frac{1}{4}+\frac{\theta}{3}+\frac{\e}{3}$, then
 \begin{align}
  \|\LR D^{-1}(\cN\cU)_{\alpha L}\|_{L^{\tilde q'}_tL_x^{\tilde r'}} \les& \|\cN\|_{L^2_t\dot B^{-1/4-\e}_{q(-\e)}}\|\cU\|_{X|Y}.
\end{align}
(2) For any $\cU$, the following estimate holds
\begin{align} \|D(\cU \bar{\cU})_{HH}\|_{L^1_tL^2_x} \lec& \|\cU\|^2_{L_t^2(\dot B^{1/4+\e}_{q(\e)}|
 B^{2/3}_{q(\e)})}.\end{align}
If $0\leq \theta\leq 1$, $\frac{1}{\tilde
q}=\frac{1}{2}-\frac{\theta}{2}$, $\frac{1}{\tilde
r}=\frac{1}{4}+\frac{\theta}{3}-\frac{\e}{3}$, then
\begin{align}
 \|D(\cU \bar{\cU})_{\alpha L+L\alpha }\|_{L^{\tilde q'}_tL_x^{\tilde r'}}\lec& \|\cU\|^2_{X|Y}.
\end{align}
\end{lem}
\begin{proof}
 (1) For the first inequality, it suffices to prove
\[ \|(\cN\cU)_{LH}\|_{L^2_x} \lec \|\cN\|_{\dot B^{-1/4-\e}_{q(-\e)}}
\|\cU\|_{L_t^2\dot B^{1/4+\e}_{q(\e)}}.\] By dyadic decomposition,
we have $(\cN\cU)_{LH}=\sum_{k_1\leq k_2-k_\alpha}P_{k_1}\cN
P_{k_2}\cU$. Then by H\"older inequality, we get
\begin{align*}
\norm{(\cN\cU)_{LH}}_{L_x^2}
\leq& (\sum_{k_2}|\sum_{k_1\leq k_2-k_\alpha}\norm{P_{k_1}\cN}_{L_x^{q(-\e)}}\norm{P_{k_2}\cU}_{L_x^{q(\e)}}|^2)^{1/2}\\
\les& (\sum_{k_2}|\sum_{k_1\leq k_2-k_\alpha}2^{k_1(\frac{1}{4}+\e)}2^{k_1(-\frac{1}{4}-\e)}\norm{P_{k_1}\cN}_{L_x^{q(-\e)}}\norm{P_{k_2}\cU}_{L_x^{q(\e)}}|^2)^{1/2}\\
\les&\|N\|_{\dot B^{-1/4-\e}_{q(-\e)}}\|u\|_{\dot
B^{1/4+\e}_{q(\e)}}.
\end{align*}
Similarly,  we can get the second one. For the third inequality, by
H\"{o}lder inequality and Sobolev embedding,
\begin{align*}
&\|(\cN \cU)_{\alpha L}\|_{L^{\tilde q'}_t L^{\tilde r'}_x}
 \lec \|\cN_\alpha\|_{L^2_t L^{q(-\e)}_x}\|\cU\|_{L^{\frac{2}{\theta}}_t L^{\frac{6}{3-2\theta}}_x}
  \lec \|\cN_\alpha\|_{L^2_t L^{q(-\e)}_x}\|\cU\|_{X}.
\end{align*}

(2) For the first inequality,
we have
\begin{align*}
 \|D(\cU \bar{\cU})_{HH}\|_{L^2_x} \leq& \sum_{|k_1-k_2|<k_\alpha}2^{k_2}\norm{P_{k_1}\cU}_{L_x^{q(-\e)}}\norm{P_{k_2}\bar{\cU}}_{L_x^{q(\e)}}\\
\les&\|\cU\|_{\dot B^{1/2-\e}_{q(-\e)}}\|\cU\|_{\dot
B^{1/2+\e}_{q(\e)}}\lec \|\cU\|^2_{\dot B^{1/4+\e}_{q(\e)}|
 B^{2/3}_{q(\e)}}.
\end{align*}
The proof of the second one is similar with the third one in (1).
\end{proof}

\subsection{Boundary terms}
Next, we estimate the boundary terms.
\begin{lem}
For any $\cN_0$ and $\cU_0$, we have
\begin{align*}
\|\LR D^{-1}\Om(\cN_0,\cU_0)\|_{H^{1}_x} \les \|\cN_0\|_{L^2_x}\|\cU_0\|_{H^1_x},\,
\|D\tilde\Om(\cU_0,\cU_0)\|_{L^2_x}\les \|\cU_0\|^2_{H^1_x}.
\end{align*}
As a consequence, for any $\cN$ and $\cU$
\begin{align*}
\|\LR D^{-1}\Om(\cN,\cU)\|_{L_t^\I H^{1}_x} \les \|\cN\|_{L_t^\I L^2_x}\|\cU\|_{L_t^\I
H^1_x},\,
\|D\tilde\Om(\cU,\cU)\|_{L_t^\I L^2_x}\les \|\cU\|^2_{L_t^\I
H^1_x}.
\end{align*}
\end{lem}
\begin{proof}
We only prove $\|\Om(\cN_0,\cU_0)\|_{L^{2}_x} \lec
\|\cN_0\|_{L^2_x}\|\cU_0\|_{H^1_x}$, since the others are similar.
From the Plancherel equality we have
\begin{align*}
\|\Om(\cN_0,\cU_0)\|_{L^{2}_x}&\les
\normo{\int_{|\xi-\eta|\gg|\eta|}|\eta|^{-1}
|\hat{N_0}(\xi-\eta)|\cdot|\hat{\cU_0}(\eta)|d\eta}_{L^2}\les
\|\cN_0\|_{L^2_x}\|\cU_0\|_{H^1_x}.
\end{align*}
where we used the Sobolev embedding
$\norm{\ft^{-1}|\x|^{-1}|\hat{u_0}(\xi)|}_{L^\I}\les
\norm{u_0}_{H^1}$.
\end{proof}

To handle the other component, we will need a Coifman-Meyer type
bilinear multiplier estimates (see  Lemma 3.5 in \cite{GN}).

\begin{lem} For any $\cN$ and $\cU$ we have
\begin{align*}
\|\LR{D}^{-1}\Om(\cN,\cU)\|_{L_t^2(\dot
B^{1/4+\e}_{q(\e)}|B^{2/3}_{q(\e)})}
  &\lec \|\cN\|_{L^\I_t L^2_x}\|\cU\|_{L_t^2L_x^6},\\
\|D\tilde\Om(\cU,\cU)\|_{L^2_t \dot B^{-1/4-\e}_{q(-\e)}}&\lec \|\cU\|_{L^\I_tH^1_x}\|\cU\|_{L_t^2L_x^6}.
 \end{align*}
\end{lem}
\begin{proof}
For the first inequality, it suffices to prove
\begin{align*}\|\Om(\cN,\cU)\|_{\dot B^{1/4+\e}_{q(\e)}|B^{-1/3}_{q(\e)}}
  \lec \|\cN\|_{L^2_x}\|\cU\|_{L_x^6}.
\end{align*}
By Sobolev embedding, we get
\begin{align*}
\|\Om(\cN,\cU)\|_{\dot B^{1/4+\e}_{q(\e)}}
\lesssim &
\|D\Om(\cN,\cU)\|_{L_x^2}.
\end{align*}
It is easy to
see that $D\Om(\cN,\cU)$ is a bilinear multiplier with the symbol
\[m(\xi,\eta)=\frac{|\xi+\eta| \sum\chi_{\leq k-5}(\eta)\chi_k(\xi)}{-\LR{\xi+\eta}+\alpha
|\xi|+\LR{\eta}},\] and $m$ satisfies the condition in Lemma 3.5 in
\cite{GN}. Thus applying dyadic decomposition and Bernstein
inequality, we get
\begin{align*}
\|P_{<0}D\Om(\cN,\cU)\|_{L_x^2}
\les& (\sum_{k_2<2}\norm{\sum_{k_1\leq k_2-k_\alpha}D\Om(P_{k_2}\cN,P_{k_1}\cU)}_{L_x^2}^2)^{1/2}\\
\les& (\sum_{k_2<2}(\sum_{k_1\leq
k_2-k_\alpha}\norm{P_{k_2}\cN}_{L_x^2}\norm{P_{k_1}\cU}_{L_x^\I})^2)^{1/2}\\
\les& (\sum_{k_2<2}(\sum_{k_1\leq
k_2-k_\alpha}(2^{k_1})^{\frac{1}{2}}\norm{P_{k_2}\cN}_{L_x^2}\norm{P_{k_1}\cU}_{L_x^6})^2)^{1/2}
\les \|\cN\|_{L^2_x}\|\cU\|_{L^6_x}.
\end{align*}
Similarly,
\begin{align*}
\|P_{\geq 0}\Om(\cN,\cU)\|_{B^{-1/3}_{q(\e)}}\lec &\|P_{\geq 0}\LR{D}^{\frac{5}{12}-\e}\Om(\cN,\cU)\|_{L_x^2}\\
\les& (\sum_{k_2\geq-2}\norm{\sum_{k_1\leq k_2-k_\alpha}\LR{D}^{\frac{5}{12}-\e}\Om(P_{k_2}\cN,P_{k_1}\cU)}_{L_x^2}^2)^{1/2}\\
\les& (\sum_{k_2\geq-2}(\sum_{k_1\leq
k_2-k_\alpha}\LR{2^{k_2}}^{-\frac{7}{12}-\e}\norm{P_{k_2}\cN}_{L_x^2}\norm{P_{k_1}\cU}_{L_x^\I})^2)^{1/2}\\
\les& (\sum_{k_2\geq-2}(\sum_{k_1\leq
k_2-k_\alpha}(2^{k_1})^{\frac{1}{2}}\LR{2^{k_2}}^{-\frac{7}{12}-\e}\norm{P_{k_2}\cN}_{L_x^2}\norm{P_{k_1}\cU}_{L_x^6})^2)^{1/2}
\\
\les&
\|\cN\|_{L^2_x}\|\cU\|_{L^6_x}.
\end{align*}
We proved the desired result.

Similarly, for the second inequality, by Sobolev embedding we get
\[\|D\tilde\Om(\cU,\cU)\|_{\dot B^{-1/4-\e}_{q(-\e)}}
  \lec \|{D}^{3/2}\tilde\Om(\cU,\cU)\|_{L_x^2}\]
and $D\tilde\Om$ behaves similarly to $D\Om$. Then applying
dyadic decomposition and Bernstein inequality, we get
\begin{align*}
\|D^{3/2}\tilde\Om(\cU,\cU)\|_{L_x^2}\les& (\sum_{k_2}\norm{\sum_{k_1\leq
k_2-k_\alpha}D^{3/2}\tilde\Om(P_{k_2}\cU,P_{k_1}\cU)}_{L_x^2}^2)^{1/2}\\
\les& (\sum_{k_2}(\sum_{k_1\leq
k_2-k_\alpha}2^{k_2/2}\norm{P_{k_2}\cU}_{L_x^2}\norm{P_{k_1}\cU}_{L_x^\I})^2)^{1/2}\\
\les& (\sum_{k_2}(\sum_{k_1\leq
k_2-k_\alpha}2^{(k_1+k_2)/2}\norm{P_{k_2}\cU}_{L_x^2}\norm{P_{k_1}\cU}_{L_x^6})^2)^{1/2}
\les \|\cU\|_{H^1_x}\|\cU\|_{L^6_x}.
\end{align*}
Thus we finish the proof of the lemma.
\end{proof}

\subsection{Cubic terms}
Finally, we deal with the cubic terms.
\begin{lem} For any $N$ and $u$ we have
\begin{align*}
\|\jb{D}^{-1}\Om(D|\cU|^2,\cU)\|_{L^1_tH^1_x}\les&
\|\cU\|_{L^2_tL_x^6}^2\|\cU\|_{L^\I_tH^1_x},\\
\|\jb{D}^{-1}\Om(\cN,\jb{D}^{-1}(\cN\cU))\|_{L^2_t(L_x^{6/5}|B^{1+5/6}_{6/5})}\les&
  \|\cN\|^2_{L^\I_tL_x^2}\|\cU\|_{L^2_tL^6_x},\\
 \|D\tilde\Om(\jb{D}^{-1}(\cN\cU),\cU)\|_{L^1_tL^2_x}\les&
 \|\cN\|_{L^\I_tL^2_x}\|\cU\|^2_{L^2_tL^6_x}.
\end{align*}
\end{lem}
\begin{proof}
As in the proof of the previous lemma, applying dyadic
decomposition, we get
\begin{align*}
\|\Om(D|\cU|^2,\cU)\|_{L^2_x}\les&(\sum_{k_2}\norm{\sum_{k_1\leq k_2-k_\alpha}\Om(P_{k_2}D|\cU|^2,P_{k_1}\cU)}_{L_x^2}^2)^{1/2}\\
\les& (\sum_{k_2}(\sum_{k_1\leq
k_2-k_\alpha}\norm{P_{k_2}|\cU|^2}_{L_x^2}\norm{P_{k_1}\cU}_{L_x^\I})^2)^{1/2}\\
\les&\||\cU|^2\|_{H^{1/2}_x}\|\cU\|_{L_x^6}
\lesssim\|\cU\|_{H^1_x}\|\cU\|_{L_x^6}^2.
\end{align*}
Similarly, for the second inequality, we have
\begin{align*}
&\|\jb{D}^{5/6}\Om(\cN,\jb{D}^{-1}(\cN\cU))\|_{L^{6/5}_x}\\\les&(\sum_{k_2}\norm{\sum_{k_1\leq k_2-k_\alpha}\jb{D}^{5/6}\Om(P_{k_2}\cN,P_{k_1}\jb{D}^{-1}(\cN\cU))}_{L^{6/5}_x}^2)^{1/2}\\
\les& (\sum_{k_2}(\sum_{k_1\leq
k_2-k_\alpha}2^{-k_2}\LR{2^{k_2}}^{5/6}\LR{2^{k_1}}^{-1}\norm{P_{k_2}\cN}_{L_x^2}\norm{P_{k_1}(\cN\cU)}_{L_x^{3}})^2)^{1/2}\\
\les& (\sum_{k_2}(\sum_{k_1\leq
k_2-k_\alpha}2^{k_1-k_2}\LR{2^{k_2}}^{5/6}\LR{2^{k_1}}^{-1}\norm{P_{k_2}\cN}_{L_x^2}\norm{P_{k_1}(\cN\cU)}_{L_x^{3/2}})^2)^{1/2}\\
\les&\|\cN\|^2_{L_x^2}\|\cU\|_{L_x^{6}},
\end{align*}
and for  the last inequality, we have
\begin{align*}
& \|D\tilde\Om(\jb{D}^{-1}(\cN\cU),\cU)\|_{L^2_x}\\\les&(\sum_{k_2}\norm{\sum_{k_1\leq k_2-k_\alpha}D\tilde\Om(\jb{D}^{-1}(\cN\cU),\cU)}_{L^{2}_x}^2)^{1/2}\\
\les& (\sum_{k_2}(\sum_{k_1\leq
k_2-k_\alpha}2^{k_1/2}\LR{2^{k_2}}^{-1}\norm{P_{k_2}(\cN\cU)}_{L_x^2}\norm{P_{k_1}\cU}_{L_x^{6}})^2)^{1/2}\\
\les& (\sum_{k_2}(\sum_{k_1\leq
k_2-k_\alpha}2^{(k_1+k_2)/2}\LR{2^{k_2}}^{-1}\norm{P_{k_2}(\cN\cU)}_{L_x^{3/2}}\norm{P_{k_1}\cU}_{L_x^{6}})^2)^{1/2}\\
\les& \norm{\cN\cU}_{L_x^{3/2}}\norm{\cU}_{L_x^{6}}
\les\|\cN\|_{L_x^2}\|\cU\|^2_{L_x^{6}}.
\end{align*}
\end{proof}

\section{Proof of Theorem \ref{thm}}

Now we are ready to use the estimates obtained in the previous section to prove Theorem \ref{thm}. For any
$(u_0,u_1,n_0,n_1)\in H^1_r(\R^3)\times L^2_r(\R^3)\times L^2_r(\R^3)\times \dot
H^{-1}_r(\R^3)$, we define an operator $\Phi_{u_0,u_1,n_0,n_1}(\cU,\cN)$ by
the right-hand side of \eqref{int-U}-\eqref{int-N}. Our
resolution space is
\[S_\eta=\{(\cU,\cN): \norm{(\cU,\cN)}_S=\norm{\cU}_{X|Y}+\norm{\cN}_{L^\I_tL^2_x \cap L^2_t\dot B^{-1/4-\e}_{q(-\e),2}}\leq \eta\}\]
endowed with the norm metric $\norm{\cdot}_S$.

We will show that $\Phi_{u_0,u_1,n_0,n_1}:S_\eta\to S_\eta$ is a
contraction mapping, provided that $\eta\ll 1$ and $(u_0,u_1,n_0,n_1)$
are sufficiently small. By the estimates in the previous section, we
have for any $(\cU,\cN)\in S_\eta$
\begin{align*}
\norm{\Phi_{u_0,u_1,n_0,n_1}(\cU,\cN)}_S\les&
\norm{\cU_0}_{H_x^1}+\norm{\cN_0}_{L_x^2}+(\norm{\cU_0}_{H_x^1}+\norm{\cN_0}_{L_x^2})^2\\
&+\norm{(\cU,\cN)}_S^2+\norm{(\cU,\cN)}_S^3\leq \eta
\end{align*}
if $\e_0=\norm{\cU_0}_{H_x^1}+\norm{\cN_0}_{L_x^2}=\norm{u_0}_{H_x^1}+\norm{u_1}_{L_x^2}+\norm{n_0}_{L_x^2}+\norm{n_1}_{\dot
H_x^{-1}}\ll 1$, and we set $\eta=C\e_0$. Similarly, we can prove
$\Phi_{u_0,u_1,n_0,n_1}:S_\eta\to S_\eta$ is a contraction mapping. Our
estimates are time global, therefore Theorem \ref{thm} follows
immediately.

\subsection*{Acknowledgment}
Z. Guo is supported in part by NNSF of China (No. 11001003) and RFDP
of China (No. 20110001120111).
S. Wang is supported by China Scholarship Council.

\end{document}